\documentclass[twoside,a4paper,12pt,centertags]{amsart}
\usepackage{amsmath,amssymb,verbatim,vmargin}
\usepackage[bookmarks=true]{hyperref}   

\usepackage{t1enc}
\usepackage{graphicx}
\usepackage{wrapfig}

\theoremstyle{plain}
\newtheorem{thm}{Theorem}[section]
\newtheorem{lem}[thm]{Lemma}

\newtheorem{prop}[thm]{Proposition}

\theoremstyle{definition}
\newtheorem{defn}[thm]{Definition}
\newtheorem{rem}[thm]{Remark}

\title[Non unique solutions]{Non unique solutions to boundary value problems for non symmetric divergence form equations
}
\author{Andreas Axelsson}
\address{Andreas Axelsson, Matematiska institutionen, Stockholms universitet, 106 91 Stockholm, Sweden}
\email{andax@math.su.se}

\mathchardef\semic="303B
\newcommand{\Mcc}{{M\raise.55ex\hbox{\lowercase{c}}}}

\newcommand{\R}{{\mathbf R}}
\newcommand{\C}{{\mathbf C}}

\newcommand{\mL}{{\mathcal L}}

\DeclareMathOperator{\re}{Re}
\newcommand{\im}{\text{{\rm Im}}\,}
\newcommand{\sett}[2]{ \{ #1 \, \semic \, #2 \} }

\newcommand{\dom}{\textsf{D}}

\newcommand{\s}{\text{{\rm sgn}}}
\newcommand{\barint}{\mbox{$ave \int$}}
\newcommand{\divv}{{\text{{\rm div}}}}
\newcommand{\curl}{{\text{{\rm curl}}}}

\newcommand{\hut}[1]{\check #1}

\newcommand{\wt}{\widetilde}

\newcommand{\pd}{\partial}

\newcommand{\eps}{\epsilon}
\newcommand{\pv}{\text{p.v.}\!}

\def\barint_#1{\mathchoice
            {\mathop{\vrule width 6pt
height 3 pt depth -2.5pt
                    \kern -8.8pt
\intop}\nolimits_{#1}}%
            {\mathop{\vrule width 5pt height
3 pt depth -2.6pt
                    \kern -6.5pt
\intop}\nolimits_{#1}}%
            {\mathop{\vrule width 5pt height
3 pt depth -2.6pt
                    \kern -6pt
\intop}\nolimits_{#1}}%
            {\mathop{\vrule width 5pt height
3 pt depth -2.6pt
          \kern -6pt \intop}\nolimits_{#1}}}

\begin{document}
\begin{abstract}
We calculate explicitly solutions to the Dirichlet and Neumann boundary value problems in the 
upper half plane, for a family of divergence form equations with non symmetric coefficients with a jump
discontinuity.
It is shown that the boundary equation method and the Lax--Milgram method for constructing
solutions may give two different solutions when the coefficients are sufficiently non symmetric.
\end{abstract}
\maketitle

%
%
%
%
\section{Introduction}

Recently, new techniques in harmonic analysis have been used to study boundary value problems (BVP's) for divergence form elliptic equations with non symmetric, or more general complex coefficients.
In the half plane, for real but non symmetric coefficients, $L_p$ solvability of the Dirichlet problem for sufficiently large $p$ was obtained by 
Kenig, Koch, Pipher and Toro~\cite{KKPT} and
$L_p$ solvability of the Neumann and regularity problems, for sufficiently small $p$, was proved 
by Kenig and Rule~\cite{KR}.
In $\R^n$, two boundary equation methods have been studied by
Alfonseca, Auscher, Axelsson, Hofmann and Kim~\cite{AAAHK} and by
Auscher, Axelsson and Hofmann~\cite{AAH} where, among other things, 
these BVP's are proved to be well posed in $L_2$ for small complex $L_\infty$ perturbations
of real symmetric coefficients. 

Unlike the case of real symmetric coefficients, for general non symmetric coefficients, well posedness
of these classical BVP's may fail.
In \cite{KKPT} and \cite{KR}, the family
$$
  A_k(x) :=
     \begin{bmatrix}
       1 & k\s(x) \\
       -k \s(x) & 1
     \end{bmatrix}
$$
of non symmetric coefficient matrices with a jump at $x=0$ was studied and shown to provide 
counter examples to well posedness for certain values of the parameter
$k\in \R$.
More precisely, the following theorem was proved in
\cite[Theorem (3.2.1)]{KKPT} and \cite[Appendix]{KR}.
\begin{thm}   \label{thm:kkprt}
  Let $1<p<\infty$.
  The Dirichlet problem (Dir-$A_k,p$) fails to be well posed in the $\dot H^1$ sense if $k> \tan(\tfrac \pi{2q})$, where $1/q= 1-1/p$.
  The regularity problem (Reg-$A_k,p$) fails to be well posed in the $\dot H^1$ sense if $k< -\tan(\tfrac \pi{2p})$.
  The Neumann problem (Neu-$A_k,p$) fails to be well posed in the $\dot H^1$ sense if $k> \tan(\tfrac \pi{2p})$. 
\end{thm}
In this paper, we demonstrate that one must be careful in specifying in what sense well posedness is meant, 
when considering BVP's for non symmetric coefficients.
(The notion of well posedness in the $\dot H^1$ sense is defined below.)
Indeed, Theorem~\ref{thm:main} below shows that these BVP's can be well posed in an $L_\infty(L_p)$ sense,
as defined below, without being well posed in the mentioned $\dot H^1$ sense.

To explain these results, we first need to introduce the notion of $\dot H^1$ solutions and 
$L_\infty(L_p)$ solutions to BVP's.
We consider a given divergence form equation 
\begin{equation}  \label{eq:divform}
  \divv A(x) \nabla U(t,x) =0
\end{equation}
in the upper half plane $\R^2_+ := \sett{(t,x)\in\R^2}{t>0}$, where 
$A= (a_{ij})_{i,j=0,1} \in L_\infty(\R; \mL(\C^2))$ is a $t$-independent, complex and accretive coefficient matrix
such that $\re(A(x)v,v)\ge \kappa |v|^2$, $x\in\R$, $v\in \C^2$, for some $\kappa>0$,
and where $U$ satisfies one of the following prescribed boundary conditions.
\begin{itemize}
\item 
The Dirichlet problem (Dir-$A,p$): $U(0,\cdot)= u$, for a given function $u\in L_p(\R;\C)$.
\item 
The (Dirichlet) regularity problem (Reg-$A,p$): 
$\partial_1U(0,\cdot) = u'$, for a given function $u\in \dot W^1_p(\R;\C)$.
\item
The Neumann problem (Neu-$A,p$):
$a_{00} \partial_0U(0,\cdot)+ a_{01}\partial_1 U(0,\cdot)= \phi$, for a given function $\phi\in L_p(\R;\C)$.
This means that the conormal derivative of $U$ is prescribed. 
\end{itemize}
Throughout this paper, $p$ denotes a fixed exponent such that $1<p<\infty$, and $q$ is the dual exponent.
The regularity and Neumann problems can be thought of as BVP's for the gradient
vector field $F(t,x)= F_0 e_0 + F_1 e_1:= \nabla U(t,x)$, rather than $U$ itself.
Here $e_0$ denotes the vertical basis vector along the $t=x_0$-axis, and $e_1$ is the horizontal basis vector along the $x=x_1$-axis.

There are two classical methods for constructing a solution $U$: Lax--Milgram's lemma and
boundary equation methods.

\vspace{2mm}
\noindent $\dot H^1$ SOLUTIONS

To solve a Dirichlet problem for a sufficiently smooth and localised boundary function $u(x)$, one 
first constructs a function $U_1$ in $\R^2$ such that $U_1(0,x)=u(x)$. Next, one uses Lax--Milgram's lemma to find a function $U_2 \in \dot H^1_0(\R^2_+)$, which decays at infinity, such that
$B(U_2,\psi)= \ell(\psi)$ for all $\psi\in \dot H^1_0(\R^2_+)$, where 
$$
 B(U_2,\psi):= \iint_{\R^2_+}  \big(A(x)\nabla U_2(t,x),\nabla \psi(t,x)\big)\, dtdx
$$
and the given functional is 
$\ell(\psi):= \iint_{\R^2_+}  \big(A(x)\nabla U_1(t,x),\nabla \psi(t,x)\big)\, dtdx$.
Details of this construction for the unbounded domain $\R^2_+$ are found in \cite[Lemma 1.1]{KR}.
The function $U:= U_1-U_2$ now solves equation (\ref{eq:divform}) and has boundary trace $u$,
in a weak sense.

We say that the BVP (Dir-$A,p$) is well posed in the $\dot H^1$ sense, if for all sufficiently smooth and localised $u$, the solution $U$ constructed above has quantitative bounds
\begin{equation}   \label{est:ntdir}
   \| N_*(U) \|_{L_p(\R)} \le C_p \| u \|_{L_p(\R)}.
\end{equation}
Similarly, (Reg-$A,p$) is well posed in the $\dot H^1$ sense, if for all sufficiently smooth and localised 
$u$, the solution $U$ constructed above has quantitative bounds
$$
   \| \wt N_*(\nabla U) \|_{L_p(\R)} \le C_p \| u' \|_{L_p(\R)}.
$$
Here 
$N_*(U)(x_0):= \sup_{|x-x_0|<t}|U(t,x)|$ 
and 
$\wt N_*(F)(x_0):= \sup_{|x-x_0|<t} t^{-1} \|F\|_{L_2(Q(t,x))}$,
where $Q(t,x)$ denotes the square centered at $(t,x)$ with sidelength $t$,
are the standard (modified) non-tangential maximal functions, and $C_p$ denotes a constant independent of $u$.

Turning to the Neumann problem, this is solved for a sufficiently smooth and localised 
boundary function $\phi(x)$ with $\int \phi=0$, by applying the Lax--Milgram lemma to obtain $U\in \dot H^1(\R^2_+)$,
such that $B(U,\psi)= \ell(\psi)$ for all $\psi\in \dot H^1(\R^2_+)$, where
$$
  \ell(\psi)= -\int_{\R}  \phi(x) \psi(x)\, dx.
$$
Details of this construction for the unbounded domain $\R^2_+$ are found in \cite[Lemma 1.2]{KR}.
This function $U$ solves equation (\ref{eq:divform}) and has conormal derivative $\phi$ at the boundary, in a weak sense.

We say that the BVP (Neu-$A,p$) is well posed in the $\dot H^1$ sense, if for all sufficiently smooth and localised $\phi$ with $\int\phi=0$, the solution $U$ constructed above has quantitative bounds
$$
   \|\wt N_*(\nabla U) \|_{L_p(\R)} \le C_p \| \phi \|_{L_p(\R)}.
$$

\vspace{2mm}
\noindent $L_\infty(L_p)$ SOLUTIONS

A different method for constructing a solution $U$ to one of the BVP's above is the boundary equation method. We are given a kernel function $K(t,x;y)$, which for each $y\in\R$ satisfies the equation 
(\ref{eq:divform}) in the variable $(t,x)\in\R^2_+$. From this we obtain, for each auxiliary function $h(y)$ on the 
boundary, a function
$$
  U(t,x) := \int_\R K(t,x;y) h(y)\, dy
$$
solving the equation (\ref{eq:divform}) in $\R^2_+$.
Taking the appropriate trace of $U$, depending on which boundary condition $U$ is supposed to satisfy, we get an equation $g= T(h)$, where $g$ denotes either $u$, $u'$ 
or $\phi$.
If the operator $T: L_p(\R)\rightarrow L_p(\R)$ is an isomorphism, then we can
solve the equation for $h$ and from this construct a solution $U$.

Obviously there is a freedom of choice for the kernel function $K(t,x;y)$. In this paper, we shall use the 
boundary equation method from Auscher, Axelsson and Hofmann~\cite{AAH}.
The Cauchy integral method used here for the Neumann and regularity problems actually uses a 
vector valued kernel $K(t,x;y)$, and constructs
the gradient vector field $F=\nabla U$ rather than $U$ itself.
With some abuse of notation (as the invertibility of the boundary equation may depend on the choice of kernel $K$), we shall say that the BVP's are well posed in the $L_\infty(L_p)$ sense, referring to the norm $\sup_{t>0}\|U(t,\cdot)\|_p$ for solutions, if this method gives rise to an
$L_p$ invertible boundary equation.

In this paper we shall prove the following surprising, in view of Theorem~\ref{thm:kkprt}, result.
\begin{thm}  \label{thm:main}
The boundary equation method of \cite{AAH} yields the following result for coefficients $A_k$.

The Dirichlet problem (Dir-$A_k,p$) is well posed in the $L_\infty(L_p)$ sense if $k\ne \tan(\tfrac \pi{2q})$,
where $1/q=1-1/p$.
In this case, the solution $U_t(x)=U(t,x)$ has bound $\|N_*(U)\|_p \le C\|u\|_p$ and 
convergence $\|U_t-u\|_p\rightarrow 0$ when $t\rightarrow 0^+$.

  The regularity problem (Reg-$A_k,p$) is well posed in the $L_\infty(L_p)$ sense if $k\ne -\tan(\tfrac \pi{2p})$.
In this case, the solution $F_t(x)=F(t,x)$ has bound $\|N_*(F)\|_p \le C\|u'\|_p$ and 
convergence $\|F_t-f\|_p\rightarrow 0$ when $t\rightarrow 0^+$, where $f_1=u'$.

  The Neumann problem (Neu-$A_k,p$) is well posed in the $L_\infty(L_p)$ sense if $k\ne \tan(\tfrac \pi{2p})$. 
In this case, the solution $F_t(x)=F(t,x)$ has bound $\|N_*(F)\|_p \le C\|\phi\|_p$ and 
convergence $\|F_t-f\|_p\rightarrow 0$ when $t\rightarrow 0^+$, where $f_0+k\s(x) f_1=\phi$.
\end{thm}
The reason for these seemingly contradictory results is that the two methods construct different solutions,
for some $k$. 
To illustrate this phenomenon, we study in detail the solutions to the Dirichlet problem
in section~\ref{sec:harmmeas}.
We derive the following explicit expression for the solution to the Dirichlet problem with the boundary equation method.
\begin{equation}  \label{eq:harmmeas}
  U(t,x)= \frac 1\pi\int_\R \frac{2xty+ |y|^{-\alpha}\im\big\{ (|x|+it)^{\alpha+1}  \big( y^2-(|x|-it)^2\big)  \big\} }
  {(t^2+(x-y)^2)(t^2+(x+y)^2)} u(y) dy,
\end{equation}
where $(t,x)\in\R_+^2$.
Here $\tan(\pi\alpha/2)=k$, and $\alpha\in (1/q-2,1/q)$ is the branch obtained with the boundary equation method. 
Denote the Poisson kernel in (\ref{eq:harmmeas}) by $P_\alpha(t,x;y)$. 
Below the harmonic measures $P_\alpha(0.5, 1; \cdot)$ are plotted for some values of $\alpha$.
\begin{figure}[htbp]
\centerline{\includegraphics[clip,width=17cm]{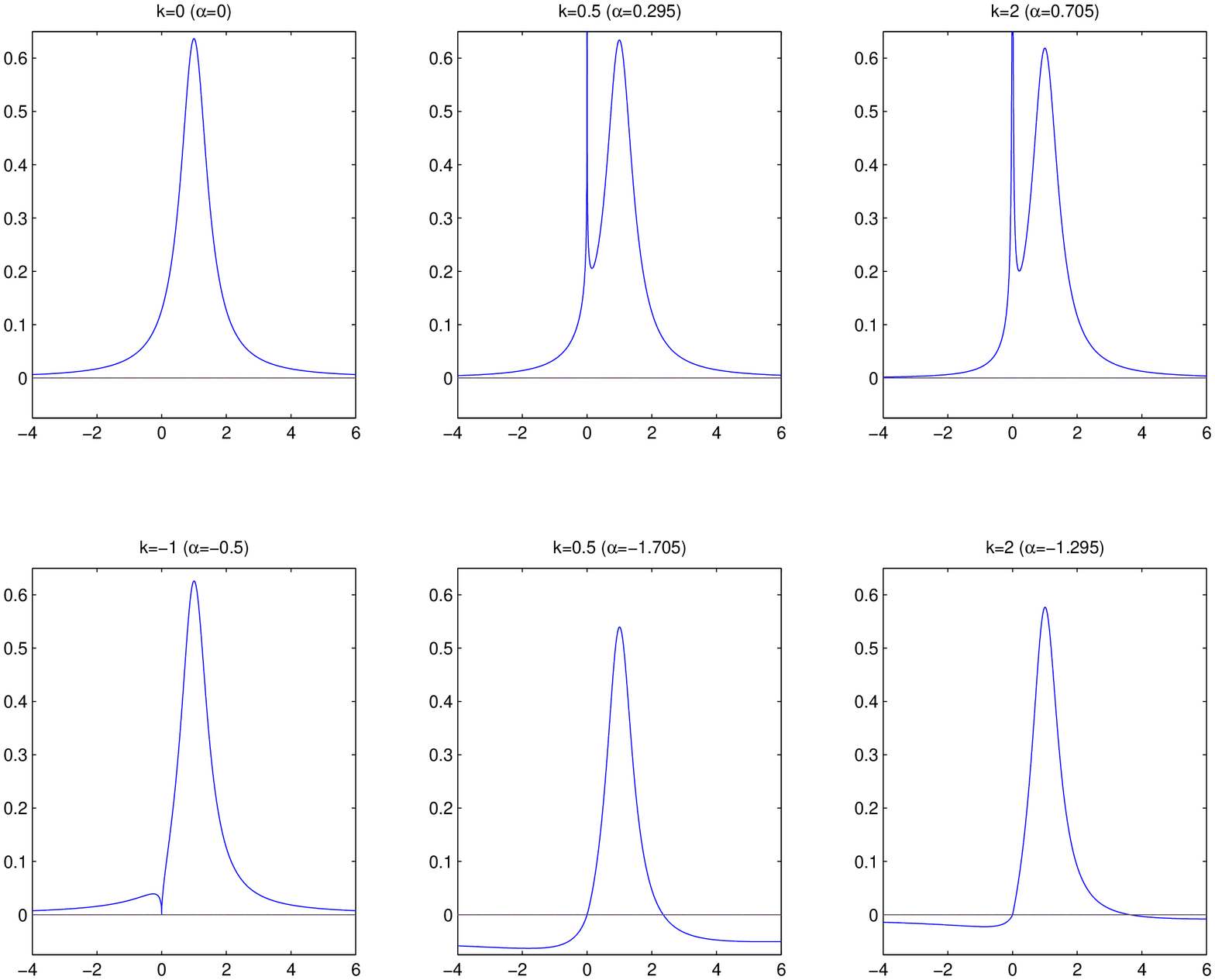}}
\end{figure}
On the other hand, for sufficiently smooth and localised boundary data, the $\dot H^1$ solution to 
the Dirichlet problem is given by (\ref{eq:harmmeas}), but with the branch 
$\alpha\in(-1,1)$. For $k> \tan(\tfrac \pi{2q})$, the $\dot H^1$ solution thus differs from the $L_\infty(L_p)$
solution, and we note the following.
\begin{itemize}
\item  The $\dot H^1$ solution uses the branch $\alpha\in(1/q,1)$. Here the kernel 
$P_\alpha(t,x;y)$ is always positive, but does not satisfy the reverse H\"older estimates $B_q$ equivalent with the estimate (\ref{est:ntdir}), since $P_\alpha(t, x;y)\sim |y|^{-\alpha}$ around $y=0$. 
The solution does not belong to $L_p(\R_x)$ for any fixed $t>0$, because of the slow decay
 $P_\alpha(t, x;y)\sim |x|^{\alpha-1}$ when $x\rightarrow\infty$.
\item 
The $L_\infty(L_p)$ solution uses the branch $\alpha\in(-2+1/q, -1)$. 
Here the kernel $P_\alpha(t,x;y)$ is not always positive, 
but does satisfy the reverse H\"older estimates $B_q$
since also $P_\alpha(t, x;y)\sim |y|^{-\alpha-2}$ when $y\rightarrow\infty$.
The solution is not $\dot H^1$ up to the boundary in a neighborhood of the origin, 
not even for smooth boundary data.
\end{itemize}
The reason why a signed harmonic measure is possible for $\alpha\in(-2+1/q, -1)$, without contradicting the maximum principle, is that the solution, even for $0\le u\in C_0^\infty(\R)$, satisfies
$\lim_{t\rightarrow 0^+} U(t,0)=-\infty$. This prevents the $L_\infty$ approximation
needed for applying the maximum principle. 
Indeed, (\ref{eq:harmmeas}) shows that
\begin{equation}  \label{eq:harmmeasimag}
  U(t,0)= \frac {\cos(\pi\alpha/2)}\pi\int_\R \frac{t^{1+\alpha}|y|^{-\alpha}}
  {t^2+y^2} u(y) dy.
\end{equation}
We summarise the main point of this paper.
When considering BVP's with non symmetric coefficients, it is important to specify which 
solution is meant. 
When the $\dot H^1$ solution to the Dirichlet problem does not satisfy (\ref{est:ntdir}), there can still exist
another solution $U_t(x)= U(t,x)$, which have bounds $\|N_*(U)\|_p <\infty$
and $L_p$ trace $\|U_t-u\|_p\rightarrow 0$, $t\rightarrow 0$.

{\bf Acknowledgments.}
The author thanks Pascal Auscher and 
Steve Hofmann for many interesting discussions on the topic of this paper.
\section{Computation of Cauchy integrals}

In this section we explicitly calculate the basic operators we need in order to solve the BVP's 
with the boundary equation method from \cite{AAH}.
As in \cite[equation (1.5)]{AAH}, we rewrite the equation (\ref{eq:divform}) for $U$ as the equivalent 
first order system $\divv A_k F=0$ and $\curl F=0$ for the vector field $F=\nabla U$.
Solving for the vertical derivative, this first order system reads
\begin{equation}   \label{eq:tsolvedeq}
  \partial_t F + 
     \begin{bmatrix}
       k(\s(x)\partial_x  -\partial_x \s(x)) & \partial_x \\
       -\partial_x & 0
     \end{bmatrix}F = 0.
\end{equation}
Throughout this paper we shall identify $f_0 e_0+ f_1 e_1 = [f_0, f_1]^t$.
The tangential matrix operator in (\ref{eq:tsolvedeq}) will be denoted $T_k$, and is seen to be 
a self-adjoint operator in $L_2(\R)$, with domain
$$
  \dom(T_k):=\sett{f_0e_0+f_1e_1}{f_0 \in H^1(\R), f_1-k\s(x) f_0\in H^1(\R)}.
$$

To solve the BVP's, we need to calculate certain operators in the functional calculus of the 
self-adjoint operator $T_k$, in particular we need the following result.
\begin{thm}   \label{thm:cauchyints}
  The Cauchy (singular) integral operators for $A_k$ are
$$
  \s(T_k) f(x) = \frac 1\pi
\begin{bmatrix}
  -\pv\int \frac{f_1(y)}{x-y} dy  \\
   \pv\int \frac{f_0(y)}{x-y} dy 
\end{bmatrix}
- \frac 1\pi\frac k{1+k^2}
\begin{bmatrix}
  \int \frac{f_0(y) + k\s(y) f_1(y)}{|x|+|y|} dy \\
  \s(x) \int \frac{kf_0(y) - \s(y) f_1(y)}{|x|+|y|} dy 
\end{bmatrix}
$$
and
\begin{multline*}
e^{-t|T_k|}\chi_+(T_k)f(x) =
\frac 1{2\pi}
\begin{bmatrix}
\int\frac{t f_0(y)- (x-y)f_1(y)}{t^2+(x-y)^2} dy \\
\int\frac{t f_1(y)+ (x-y)f_0(y)}{t^2+(x-y)^2} dy 
\end{bmatrix} \\
+\frac 1{2\pi}\frac{k}{1+k^2}
\begin{bmatrix}
-\int\frac{t(k f_0(y)-\s(y) f_1(y))+(|x|+|y|)(f_0(y)+k\s(y) f_1(y)) }{ t^2+(|x|+|y|)^2} dy \\
\s(x)\int\frac{t(f_0(y)+k\s(y) f_1(y))+(|x|+|y|)(-k f_0(y)+\s(y) f_1(y)) }{ t^2+(|x|+|y|)^2} dy 
\end{bmatrix}.
\end{multline*}
\end{thm}
Here $\chi_\pm(z)$ denotes the characteristic function of the right/left complex half plane.
We write $\s(z):= \chi_+(z)-\chi_-(z)$  and $|z|:= z\,\s(z)$, for $z\in\C$.
Note that $|z|$ does not denote absolute value for non real $z$.

\begin{lem}   \label{lem:resolvents} 
For non real $i\lambda\in\C$, the resolvent $(i\lambda - T_k)^{-1}f= u$ is given by
\begin{multline*}
\begin{bmatrix}
u_0(x) \\ u_1(x)
\end{bmatrix}
=
\frac {\s \lambda} 2
\begin{bmatrix}
\int e^{-|\lambda(x-y) |} \big( -i f_0(y) +\s(\lambda(x-y)) f_1(y) \big)  dy \\
\int e^{-|\lambda(x-y)|} \big( -\s(\lambda(x-y)) f_0(y) -i f_1(y) \big)   dy 
\end{bmatrix}
\\
+ e^{-|\lambda x|}\tfrac k{2(1- ik\s(\lambda) )}
\begin{bmatrix}
 i \int  e^{-| \lambda y|}  \big( -i f_0(y) -\s(\lambda y) f_1(y) \big)  dy \\
 \s(\lambda x) \int  e^{-| \lambda y|}  \big( -i f_0(y) -\s(\lambda y) f_1(y) \big)  dy 
\end{bmatrix}.
\end{multline*}

%
%
%
%
%
%
%
%
\end{lem}

To prove the lemma, we need to solve $(i\lambda- T_k)u=f$ for $u$. Thus we are looking for $u$ such that
$$
  \begin{cases}
    u_0' = -i\lambda u_1+f_1, \\
    u_1' = i\lambda u_0- f_0,
  \end{cases}
$$
for $x\ne 0$, and where $u\in\dom(T_k)$, i.e. $u_0$ is continuous at $x=0$, whereas 
$$
  u_1(0+) - u_1(0-) = 2k u_0(0).
$$
Multiplying the system of equations with 
$M= [-i, 1; 1, -i]$ gives the diagonal system
$$
  \begin{cases}
    v_0' = -\lambda v_0 + g_0, \\
    v_1' = \lambda v_1 +g_1,
  \end{cases}
$$
for $v= Mu$ and $g := [-1, -i;i, 1]f$.
Integrating these equations and using the jump condition at $x=0$ gives the formula in the lemma.

\begin{lem}   \label{lem:ptqt}
  The operators $P_t=(1+t^2 T_k^2)^{-1}$ and $Q_t = tT_k(1+t^2 T_k^2)^{-1}$, for $t>0$, are
\begin{align*}
P_t f(x) &= 
\begin{bmatrix}
\int\frac 1{2t} \Big( e^{-|x-y|/t} f_0(y) +  \frac k{1+k^2} e^{-(|x|+|y|)/t} \big( -kf_0(y) + \s(y) f_1(y) \big)  \Big)  dy  \\
\int\frac 1{2t} \Big( e^{-|x-y|/t} f_1(y) +  \frac {k\s(x)}{1+k^2} e^{-(|x|+|y|)/t} \big( f_0(y) + k\s(y) f_1(y) \big) \Big)  dy  
\end{bmatrix},
\\
Q_t f(x) &= 
\begin{bmatrix}
\int -\frac 1{2t} \s(x-y) e^{-|x-y|/t} f_1(y) dy  \\
\int\frac 1{2t}  \s(x-y) e^{-|x-y|/t} f_0(y)  dy  
\end{bmatrix} \\
&-
\begin{bmatrix}
\int\frac 1{2t}  \frac k{1+k^2} e^{-(|x|+|y|)/t} \big( f_0(y) + k\s(y) f_1(y) \big)   dy  \\
\int\frac 1{2t}  \frac {k\s(x)}{1+k^2} e^{-(|x|+|y|)/t} \big( kf_0(y) - \s(y) f_1(y) \big)  dy  
\end{bmatrix}.
\end{align*}
\end{lem}

This follows from Lemma~\ref{lem:resolvents} and the formulae
\begin{align*}
  P_t & = \tfrac 1{2it}  (  ( \tfrac 1{it}- T_k )^{-1} - (\tfrac 1{-it} -T_k)^{-1}  ) \text{ and} \\
  Q_t & = -\tfrac 1{2t}  (  ( \tfrac 1{it}- T_k )^{-1} + (\tfrac 1{-it} -T_k)^{-1}  ).
\end{align*}

We are now in position to prove Theorem~\ref{thm:cauchyints}.
As in \cite[Section 2.3]{AAH}, we use the Dunford functional calculus formula
$$
  b(T_k)= \frac 1{2\pi i}\int_\gamma b(\lambda) (\lambda - T_k)^{-1} d\lambda,
$$
where $\gamma$ is the boundary of a double sector around $\R\setminus\{0\}$.
Using $b(z)= \s(z)$ and $b(z)= e^{-t|z|} \chi_+(z)$ respectively, and 
Lemma~\ref{lem:resolvents}, gives the formulae in Theorem~\ref{thm:cauchyints}.
However, the computations can be somewhat simplified by choosing a degenerate
contour of integration along the imaginary axis. 
In this case the Dunford formulae become
\begin{align*}
\s(T_k) &= \frac 2\pi \int_0^\infty Q_s \frac{ds}s, \\
e^{-t|T_k|}\chi_+(T_k) &= \frac 1\pi \int_0^\infty \Big( Q_s \cos(t/s) + P_s \sin(t/s) \Big) \frac{ds}s.
\end{align*}
Changing the order of integration for $s$ and $y$ here, and using that 
$$
  \int_0^\infty \tfrac 1s e^{-x/s} e^{it/s} \tfrac{ds}s = \frac{x+ it}{x^2+t^2},
$$
gives the desired formulae.
Some of the above computations are of course formal. However, they can be justified
for example with arguments as in \cite{McQ}.

\section{Solvability of boundary equations}   \label{sec:inv}

In this section we use the Cauchy integrals from Theorem~\ref{thm:cauchyints} to solve
BVP's, following the boundary equation method described in \cite{AAH}. 
\begin{defn}
  Let $E_k^\pm h := \chi_\pm(T_k)h$ be the {\em Hardy projections}, with the
  associated {\em Cauchy singular integral operator} $E_k h:= \s(T_k)h$.
  Let the {\em Cauchy extension operators} be $(C_k^\pm h)(t,x):= (e^{\mp t|T_k|} E_k^\pm h)(x)$,
  $\pm t>0$.
\end{defn}
Note that $E_k= E_k^+ - E_k^-$ and conversely $E_k^\pm= \tfrac 12(I\pm E_k)$.
Given a function $h:\R\rightarrow\C^2$ on the boundary, applying the Cauchy extension $C_k^+$ gives
a vector field $F(t,x) = C_k^+ h(x)$ in $\R^2_+$.
This is our ansatz for the regularity and Neumann problems. 
Indeed, the vertical derivative of $F= e^{-t|T_k|} E_k^+h$ is
$$
  \pd_t F= -|T_k| e^{-t|T_k|} \chi_+(T_k) h = -T_k (e^{-t|T_k|} \chi_+(T_k) h)= -T_k F.  
$$
Thus $F$ satisfies (\ref{eq:tsolvedeq}), or equivalently the first order system 
$\divv A_k F=0$ and $\curl F=0$. 
This means that $F$ is a gradient vector field $F=\nabla U$, with potential $U$
that solves (\ref{eq:divform}).

On the other hand, to solve the Dirichlet problem, we make use of the fact that, due to the $t$-independence
of the coefficients $A_k$, we have
$$
   0= \partial_t(\divv A_k(x) F )= \divv  A_k(x) (\partial_t F)= 
      \divv A_k(x)\nabla F_0,
$$
i.e. the normal component $F_0= e_0\cdot C_k^+ h$ 
of the Cauchy extension satisfies the equation (\ref{eq:divform}).
This will be our ansatz for the Dirichlet problem.

\begin{lem}   \label{lem:cauchytrace}
  Let $h\in L_p(\R;\C^2)$.
  Then we have bounds $\|E_kh\|_p \le C \|h\|_p$ and
  $\|N_*(C_k^+h)\|_p\le C \|h\|_p$, and convergence
$$
     \| C_k^+ h(t,\cdot) - E_k^+ h\|_p \longrightarrow 0, \qquad t\longrightarrow 0^+.
$$
\end{lem}
\begin{proof}
By breaking up the integrals $\int_\R=\int_{\R_+}+ \int_{\R_-}$ for the second terms in
Theorem~\ref{thm:cauchyints}, we see that we can write each component of $C_k^+h$
as
$$
  \chi_+(x)\Big( \tfrac t{t^2+x^2}* \tilde h_1(x) +  \tfrac x{t^2+x^2}* \tilde h_2(x)  \Big)
  +  \chi_-(x)\Big( \tfrac t{t^2+x^2}* \tilde h_3(x) +  \tfrac x{t^2+x^2}* \tilde h_4(x)  \Big),
$$
for some $L_p$ functions $\tilde h_i$.
Similarly $E_k h$ is expressed in terms of the Hilbert transform. 
The lemma now follows from well known $L_p$
bounds and convergence for these convolution operators.
\end{proof}

\begin{defn}
Let the {\em double layer potential type operator} for $A_k$ be
$$
  Kf(x) := \s(x)\frac 1\pi \pv\int\frac{f(y)}{x-y} dy- \frac 1\pi\int \frac{f(y)}{|x|+|y|} dy,
$$
acting boundedly in $L_p(\R;\C)$.
\end{defn}
\begin{prop}   \label{prop:bdyinteqs}
  The boundary equation method from \cite{AAH}
constructs solutions to the BVP's as follows.
\begin{itemize}
\item 
A solution to the Neumann problem (Neu-$A_k,p$) is given by 
$$\nabla U(t,x)= C_k^+ \begin{bmatrix} \psi & 0 \end{bmatrix}^t,$$
where
$
  \phi = \tfrac 12(I+k K)\psi.
$
\item 
A solution to the regularity problem (Reg-$A_k,p$) is given by 
$$\nabla U(t,x)= C_k^+ \begin{bmatrix} -k\s(x) \psi(x) & \psi(x) \end{bmatrix}^t,$$
where
$
  \s(x) u'(x) = \tfrac 12(I- k K)(\s(y)\psi(y)).
$
\item 
A solution to the Dirichlet problem (Dir-$A_k,p$) is given by 
$$U(t,x)= e_0\cdot C_k^+ \begin{bmatrix} \psi(x) & k\s(x) \psi(x) \end{bmatrix}^t,$$
where
$
  u = \tfrac 12(I+k K)^*\psi.
$
\end{itemize}
\end{prop}
\begin{proof}
Lemma~\ref{lem:cauchytrace} shows that the
Cauchy extension has trace 
$\lim_{t\rightarrow 0^+}C_k^+ h= \tfrac 12(h+E_k h)$.
To solve the Neumann problem, we write $h= \begin{bmatrix} \psi & 0 \end{bmatrix}^t$
for the given ansatz.
Using the first formula in Theorem~\ref{thm:cauchyints} for $E_k$, 
the Neumann boundary condition becomes
\begin{multline*}
  \phi= e_0\cdot \big( A_k \,\tfrac 12(h+ E_k h) \big) \\
  =  \tfrac 12\begin{bmatrix} 1 & k\s(x)\end{bmatrix}
    \left( \begin{bmatrix} \psi \\ 0 \end{bmatrix}
    +\frac 1\pi \begin{bmatrix} 0 \\ \pv \int \tfrac{\psi(y)}{x-y} dy  \end{bmatrix}
    -\frac 1\pi\frac k{1+k^2}  
     \begin{bmatrix} \int \tfrac{\psi(y)}{|x|+|y|}dy  \\ \s(x) \int \tfrac{k\psi(y)}{|x|+|y|}dy   \end{bmatrix}
      \right)\\
    = \tfrac 12( \psi+ k K\psi ),
\end{multline*}
as stated.
Similar calculations for the Dirichlet boundary conditions give the stated boundary equations.
\end{proof}
\begin{rem}
For motivation of the choices made for the ansatzes, we refer to \cite{AAH}. In the language used
there, we here express $\hat N_A^- E^+_A \hat N^-_A$, $\hat N_A^+ E^+_A \hat N^+_A$
and $\hut N_A^- E^+_A \hut N^-_A$ respectively in terms of the operator $K$.
The idea is to compress the Cauchy integral to a suitable subspace in $L_p(\R)$, which is
complementary to the null space of the boundary condition under consideration.
We use the projections 
$\hat N_A^- =  [ 1, k\s(x) ; 0, 0 ]$,
$\hat N_A^+ =  [ 0, -k\s(x) ; 0, 1 ]$,
$\hut N_A^- =  [ 1, 0; k\s(x), 0 ]$,
which project onto the complementary subspace, along the null space.
\end{rem}
To prove Theorem~\ref{thm:main}, we see that it suffices to investigate the $L_p$
spectrum of $K$.
We first note that $K= K_0\oplus K_0$ in the splitting 
$L_p(\R)= L_p(\R_+)\oplus L_p(\R_-)$, where $K_0$ is the operator in the 
following proposition.
\begin{prop}  \label{prop:invofdlp}
The operator 
$$
  K_\alpha f(x) :=\frac 2\pi\pv\int_0^\infty \frac {x^\alpha y^{1-\alpha}}{x^2-y^2} f(y) dy
$$ 
is bounded on $L_p(\R_+;\C)$ if and only if $\alpha\in (-1/p, 2-1/p)$.
Furthermore, if  $\alpha\in (-1/p, 2-1/p)$ and $k= \tan(\tfrac \pi 2 \alpha)$, then
$I- k K_0$ is invertible in $L_p$ with inverse
$$
  (I-k K_0)^{-1} = \tfrac 1{1+k^2} (I+ k K_\alpha).
$$
\end{prop}
\begin{proof}
We use the isometry
$$
  U: L_p(\R_+)\longrightarrow L_p(\R): f(x) \longmapsto e^{t/p} f(e^t).
$$
A calculation shows that $UK_\alpha U^{-1} = \wt K_{\alpha+1/p}$ for the convolution operator
$$
  \wt K_\gamma g := \frac 2\pi\pv\frac{e^{\gamma t}}{e^{2t}-1} * g.
$$
By standard singular integral theory $\wt K_\gamma$ is bounded, on all $L_p$ spaces, 
if and only if $\gamma\in (0,2)$. This proves the boundedness result for $K_\alpha$.

To verify the inverse relation, we need to show that
$$
  (I-k\wt K_{1/p})(I+ k \wt K_{\alpha+1/p}) = 1+k^2.
$$
Applying the Fourier transform, this amounts to 
$$
  \left(  1-  ki\frac{1+z}{1-z}  \right) \left( 1+ ki\frac{1+ z e^{i\pi \alpha}}{1- z e^{i\pi\alpha}} \right)= 1+k^2,
$$
where $z:= e^{\pi(\xi + i/p)}$.
This is verified using the relation $e^{i\pi\alpha}= (1+ik)/(1-ik)$.
\end{proof}
\begin{proof}[Proof of Theorem~\ref{thm:main}]
  Proposition~\ref{prop:invofdlp} shows in particular that $I-k K_0$ is invertible in $L_p(\R)$
  if $k\ne -\tan(\tfrac \pi {2p})$. 
  Therefore the boundary equations derived in Proposition~\ref{prop:bdyinteqs} are invertible
  under the hypotheses in Theorem~\ref{thm:main}.
  The bounds and convergence of the solutions follow from Lemma~\ref{lem:cauchytrace}.
\end{proof}

\section{The harmonic measure}   \label{sec:harmmeas}

In Proposition~\ref{prop:bdyinteqs}, the kernels of the solution operators can now be calculated.
We here only calculate the harmonic measure $P_\alpha(t,x;\cdot)dy$, where $P_\alpha$ denotes the Poisson kernel, i.e. the kernel of the solution operator for the Dirichlet problem. 
The regularity and Neumann problems can be further studied in much the same way.
According to Proposition~\ref{prop:bdyinteqs} and Theorem~\ref{thm:cauchyints}, a solution to the Dirichlet problem is 
$$
  U(t,x)= \frac 1{2\pi}\int\left( \frac{t-(x-y)k\s(y)}{t^2+ (x-y)^2} - k\frac{|x|+|y|}{t^2+(|x|+|y|)^2} \right) \psi(y) dy,
$$
with
\begin{equation}  \label{eq:utoh}
  \psi(\pm y)= \frac 2{1+k^2}\left( u(\pm y)-k \frac 2\pi\pv\int_0^\infty \frac{y^{1+\alpha}z^{-\alpha}}{y^2-z^2}  u(\pm z) dz  \right), \qquad y>0,
\end{equation}
where $k=\tan(\tfrac\pi 2\alpha)$ and 
$\alpha\in(1/q-2, 1/q)$, according to Proposition~\ref{prop:invofdlp}.

Composing these two operators yields the formula (\ref{eq:harmmeas}) for the harmonic measure.
The calculation make use of the following residue calculus formula.
\begin{lem}
  For $\alpha\in(-2,1)$ and $t,z>0$, write
$$
  I_{\alpha,\beta, \gamma}(t,x,z) := \pv\int_0^\infty \frac{\gamma t+\beta(y-x)}{t^2+(y-x)^2}\frac{y^{1+\alpha}}{y^2-z^2} dy.
$$
Then, with $\arg(x+it)\in(0,\pi)$, we have
\begin{multline*}
   2 \tfrac k\pi z^{-\alpha} I_{\alpha,\beta,\gamma}(t,x,z)= 
   \frac{k^2-1}2\frac{\gamma t-\beta(x-z)}{t^2+(x-z)^2} 
   -\frac{k^2+1}2\frac{\gamma t-\beta(x+z)}{t^2+(x+z)^2} \\
   -\re\left(  \frac{(1-ik)^2(\beta-i\gamma)(x+it)^{1+\alpha}}{(x+it)^2-z^2}  \right) z^{-\alpha}.
\end{multline*}
\end{lem}
From section~\ref{sec:inv} it is clear that $u\mapsto \psi\mapsto U$ gives a solution to
(Dir-$A_k,p$) when $\alpha\in(1/q-2,1/q)$, and that $U_t\rightarrow u$ in $L_p(\R)$ when
$t\rightarrow 0^+$.
We now investigate the Poisson integral formula (\ref{eq:harmmeas}) for more general
$\alpha$.
First note that $P_\alpha(t,x;\cdot)$ is not locally $L_1$ if $\alpha\ge 1$, and if $\alpha\le -2$
then it does not decay at $\infty$. For these reasons, we only consider $\alpha\in(-2,1)$.
We shall now consider the branch $\alpha\in(-1,1)$ in (\ref{eq:harmmeas}) and 
show that this gives the $\dot H^1$ solution to (Dir-$A_k,p$).
For the rest of the section, we assume that $u\in C^\infty_0(\R)$ and $\alpha\in(-1,1)$.

From (\ref{eq:utoh}), it is seen that $\psi\in \dot H^{1/2}(\R)$ if $u\in C^\infty_0(\R)$.
Moreover, we shall use the fact that $\tilde \psi:= \psi- \psi(0) e^{-|x|}$ satisfies 
$\tilde \psi\in \dot H^{1/2}(\R)$ and has bounds 
$|\tilde\psi(x)|\le C\min(|x|, 1/|x|)^\gamma$ for some $\gamma>0$.
\begin{lem}
  If $u\in C^\infty_0(\R)$ and $\alpha\in(-1,1)$, then 
$$
  \iint_{\R^2_+} |\nabla U(t,x)|^2 dtdx <\infty.
$$
\end{lem}
\begin{proof}
Write $h:= \begin{bmatrix} \psi(x) & k\s(x) \psi(x) \end{bmatrix}^t$ for the ansatz in
Proposition~\ref{prop:bdyinteqs}, and note that
$$
  \nabla U= \nabla(e_0\cdot C_k^+ h)= \partial_t C_k^+h 
  = -T_k C_k^+ h = -E_k^+|T_k| e^{-t|T_k|}h.
$$
Since $T_k$ is self-adjoint, it satisfies $L_2$ quadratic estimates. Therefore
\begin{multline*}
    \iint_{\R^2_+} |\nabla U(t,x)|^2 dtdx =\int_0^\infty \|E_k^+|T_k| e^{-t|T_k|}h\|_2^2 dt \\
    \le \int_0^\infty \| |tT_k|^{1/2}e^{-t|T_k|}(|T_k|^{1/2}h) \|_2^2 \frac{dt}t
    \approx \| |T_k|^{1/2}h \|_2^2 \\
    \approx  \int_0^\infty \| |tT_k|^{3/2} (1+t^2T_k^2)^{-1}  (|T_k|^{1/2}h) \|_2^2 \frac{dt}t
     = \int_0^\infty \|h - P_t h \|_2^2\frac {dt}{t^2},
\end{multline*}
since $|z|^{1/2}e^{-|z|}$ and $|z|^{3/2}(1+z^2)^{-1}$ decays at $0$ and $\infty$.
Lemma~\ref{lem:ptqt} shows that
$$
   \|h - P_th \|_2^2 = \| \psi - p_t * \psi \|_2^2 + k^2 \| \s(x) \psi- p_t* (\s(x)\psi) - 2t (\psi,p_t) \s(x) p_t   \|_2^2,
$$
where $p_t(x):= \tfrac 1{2t}e^{-|x|/t}$.
We now recall that $\psi= \tilde\psi+ 2\psi(0) p_1$, where $\tilde\psi\in \dot H^{1/2}(\R)$ and
$|\tilde\psi(x)|\le C\min(|x|, 1/|x|)^\gamma$ for some $\gamma>0$, and that
$$
  \|f\|^2_{\dot H^{1/2}(\R)} \approx \int_\R\int_\R \frac{|f(x)-f(y)|^2}{|x-y|^2} dx dy
  \approx \int_\R |\hat f(\xi)|^2 |\xi| d\xi.
$$
Using the first expression for the norm, we verify that $\s(x)\tilde\psi\in \dot H^{1/2}(\R)$, and
applying Plancherel's theorem and using the second expression, we show that
$$
  \int_0^\infty \| f- p_t* f\|_2^2\frac {dt}{t^2} \le C \|f\|^2_{\dot H^{1/2}(\R)}.
$$
Thus it remains to show that 
$$
  \int_0^\infty |(\tilde\psi, p_t)|^2 \frac{dt}t + |\psi(0)|^2 \int_0^\infty \| \s(x) p_1- p_t*(\s(x) p_1)
      - 2t(p_1,p_t) \s(x) p_t  \|^2\frac {dt}{t^2} <\infty.
$$
Here the first term is finite since $|(\tilde\psi, p_t)|\le C \min(t, 1/t)^\gamma$, and an explicit 
calculation for the second term gives
$\|\cdots \|^2 \le C \min(t^2, 1)$,
which shows that the second term is finite. This proves the lemma.
\end{proof}

Restricting (\ref{eq:harmmeas}) to $x=0$ gives the formula (\ref{eq:harmmeasimag}).
In the quadrants $t>0$, $\pm x>0$, the equation $\divv A_k \nabla U=0$ reduces to 
the Laplace equation. 
For example, for the first quadrant, the Poisson integral for the Laplace equation yields
\begin{multline}  \label{eq:poissonforquadrants}
  U(t,x)= 
  \frac 1\pi \int_0^\infty \frac{4xty}{4x^2t^2+(x^2-t^2-y^2)^2} u(y) dy \\
  +\frac 1\pi \int_0^\infty \frac{4xts}{4x^2t^2+(x^2-t^2+s^2)^2} U(s,0) ds
\end{multline}
when $t>0$, $x>0$.
Inserting the expression (\ref{eq:harmmeasimag}) in the second term and calculating the integral
\begin{multline*}
  \frac 1\pi\int_0^\infty \frac{s^{2+\alpha}}{(4x^2t^2+(x^2-t^2+s^2)^2)(s^2+y^2)} ds \\
  = - \frac 12\frac 1{\cos(\pi\alpha/2)} \frac{|y|^{1+\alpha}}{4x^2t^2+(x^2-t^2-y^2)^2}
  +\frac 1{4xt}\re\left( (1-ik)\frac{(t+ix)^{1+\alpha}}{(t+ix)^2+y^2} \right),
\end{multline*}
is seen to give back (\ref{eq:harmmeas}).

It is clear from (\ref{eq:poissonforquadrants}) that $U$
satisfy the Laplace equation in the quadrants $t>0$, $\pm x>0$.
Furthermore, calculating $\partial_x U(t, 0\pm)$ from (\ref{eq:harmmeas}) and $\partial_t U(t,0)$
from (\ref{eq:harmmeasimag}), shows that 
$\pd_x U(t,0+) - \pd_x U(t,0-) = 2k \pd_t U(t,0)$ for $t>0$. 
This proves that $\divv A_k \nabla U=0$ in $\R^2_+$.
Furthermore it is clear from (\ref{eq:poissonforquadrants}) that $U$ has boundary trace $u$ in the
weak sense.
This proves that (\ref{eq:harmmeas}) gives the $\dot H^1$
solution to the Dirichlet problem when $\alpha\in (-1,1)$.
Finally we note that it follows from
(\ref{eq:poissonforquadrants}) and (\ref{eq:harmmeasimag}) that 
$P_\alpha(t,x;y)\ge 0$ for all $t>0$, $x,y\in\R$, when $\alpha\in (-1,1)$.

\bibliographystyle{acm}

\end{document}